\documentclass[11pt]{amsart}
\usepackage[T1]{fontenc}
\usepackage[cp1250]{inputenc}
\usepackage{amsmath}
\usepackage{amssymb}
\usepackage{amscd}
\usepackage[colorlinks, backref]{hyperref}
\usepackage{tikz}
\usetikzlibrary{cd}
\usepackage{clock}
\usepackage{enumerate}
\usepackage[all]{xy}

\newenvironment{pf}{\begin{proof}}{\end{proof}}

\newcommand{\w}{\operatorname{w}}

\newcommand{\ide}{\operatorname{id}}

\newcommand{\Aaa }{\mathcal A}

\newcommand{\Raa }{\mathcal R}
\newcommand{\Bee }{\mathcal B}

\newcommand{\Pee }{\mathcal P}

\newcommand{\ca}{{\mathbb{Q}}}

\newcommand{\fK}{{\mathfrak{K}}}
\newcommand{\fL}{{\mathfrak{L}}}

\newcommand{\Ob}{\operatorname{Obj}}

\newcommand{\Cod}{\operatorname{Cod}}

\newcommand{\rest}{\restriction}

\newcommand{\loe}{\leq}

\newcommand{\subs}{\subseteq}

\newcommand{\liminv}{\varprojlim}

\newcommand{\cl}{\operatorname{cl}}
\renewcommand{\int}{\operatorname{int}}
\newcommand{\fra}{Fra\"iss\'e }

\newcommand{\amor}{\SelectTips{cm}{11}}

\newcommand{\map}[3]{#1\colon #2 \to #3} 
\newcommand{\obj}[1]{\operatorname{Obj}\left(#1\right)}
\newcommand{\nic}[1]{}

\newtheorem{theorem}{Theorem}[section]
\newtheorem{corollary}[theorem]{Corollary}

\newtheorem{lemma}[theorem]{Lemma}

\newtheorem{proposition}[theorem]{Proposition}
\theoremstyle{definition}

\author{Wies{\l}aw Kubi\'s} 
\address{Wies{\l}aw Kubi\'s\\  Institute of Mathematics, Czech Academy of Sciences, Czechia\\
Institute of Mathematics\\
Cardinal Stefan Wyszynski Univwersity in Warsaw\\
Warszawa, Poland\\} \email{kubis@math.cas.cz}
\author{Andrzej Kucharski} 
\address{Andrzej Kucharski \\ University of Silesia in Katowice \\  Bankowa 14, 40-007 Katowice, Poland} \email{andrzej.kucharski@us.edu.pl} 
\author{ S\l awomir Turek}
\address{S\l awomir Turek\\ Institute of Mathematics\\
	Cardinal Stefan Wyszynski Univwersity in Warsaw\\
	Warszawa, Poland\\}
\email{s.turek@uksw.edu.pl}
\date{ver. 2024-07-07, \clocktime\today}
	
\title{On generic topological embeddings}

\begin{document}
	
	\maketitle
	
	\begin{abstract}
	 We show that an embedding of a fixed 0-dimensional compact space $K$ into the \v Cech--Stone remainder $\omega^*$ as a nowhere dense P-set is the unique generic limit, a special object in the category consisting of all continuous maps from $K$ to compact metric spaces.
	 Using  \fra theory we get a few well know theorems about  \v Cech--Stone remainder.	 
	 We establish the following:
	 
	 -- an ultrametric space $K$ of weight $\kappa$ can be uniformly embedded into $\kappa^\kappa$ as  a uniformly nowhere dense subset,
	 
	 -- every uniform homeomorphism of uniformly nowhere dense sets in  $\kappa^\kappa$ 
	 can be extended to a uniform auto-homeomorphism of  $\kappa^\kappa$,
	 
	 -- every uniformly nowhere dense set in  $\kappa^\kappa$   is a uniform retract of 	$\kappa^\kappa$.
	 
	 If we assume that $\kappa$ is a weakly compact cardinal we get the counterpart of the above result 	 without the uniformity assumption.
	\end{abstract}
\tableofcontents

\section{Introduction}

A classical theorem of Cantor says that the rational numbers $\ca$ is the only (up to isomorphism) countable linear ordering 
which contains all other countable linear orderings and every isomorphism between two finite subsets extends to 
an automorphism of the rational numbers $\ca.$ In 1927 \cite{Urysohn} Urysohn found a complete separable metric space 
$\mathbb{U}$ which has very similar properties to the rationals, considering isometries between finite subsets. Urysohn’s work 
is an analogue of Cantor’s back-and-forth argument for metric spaces. A model-theoretic approach to the back-and-forth argument 
is due to Roland \fra in 1954 \cite{F1}. Irwin \& Solecki \cite{IrwSol} presented a variant of \fra theory
with reversed arrows, i.e., epimorphisms of finite structures instead of embeddings. We are looking for generic object, a special object in a bigger category, characterized by a certain variant of projectivity. Our approach fits into the framework  of Droste \& G\"obel theory~\cite{DroGoe} and of the first author~\cite{KubFra}. We prove that an embedding $\eta\colon K\to\omega^*$ of a fixed 0-dimensional compact space $K$ into $\omega^*$ as a nowhere dense P-set is a generic object in a suitable category. From these statement, assuming CH, we derive two consequences: 
\begin{itemize}
	\item (Balcar, Frankiewicz and Mills \cite{Bal-Fra-Mills}) Every compact 0-dimensional F-space of weight $\mathfrak{c}$ can be embedded as a nowhere dense closed P-set in $\omega^*.$
 
\item (van Douwen and van Mill \cite{Mill-Douw}) Every homeomorphism of nowhere dense closed $P$-sets in  $\omega^*$ 
	can be extended to an autohomeomorphism of  $\omega^*$ and  every closed $P$-set of weight at most $\mathfrak{c}$  is a retract of 
	$\omega^*$.
	\end{itemize}

 Following \cite{Nyi} we use ultrametrics, to be defined in Section~\ref{sekcja:3}. We show that

  --  an ultrametric space $K$ of weight $\kappa$ can be  uniformly embedded into $\kappa^\kappa$ as a uniformly nowhere dense subset,
 
 --  every uniform homeomorphism of uniformly nowhere dense sets in  $\kappa^\kappa$ 
 can be extended to a uniform auto-homeomorphism of  $\kappa^\kappa$,
 
 -- every uniformly nowhere dense set in  $\kappa^\kappa$   is a uniform retract of 	$\kappa^\kappa$.
 
 In Section~\ref{sekcja:4}  

 we assume that $\kappa$ is a weakly compact cardinal. Then we get similar result to the above about the ultrametric space $2^\kappa$, which becomes $\kappa$-compact, nowhere dense subset of $\kappa$-compact ultrametric space 
 $2^\kappa$  becomes a uniformly nowhere dense subset and  a continuous map from a $\kappa$-compact ultrametric space to a Hausdorff ultrametric space becomes continuous with respect to the ultrametric.

\section{Preliminaries}

We recall some basic results concerning \fra theory  developed by
Kubiś \cite{KubFra}.	For undefined notions concerning category theory we refer to \cite{MacLane}.
Let $\mathfrak{K}$ be a category and let $\kappa$ be a cardinal. 

A category $\mathfrak{K}$ is {\em $\kappa$-complete} if every inverse sequence in $\fK$ of length $<\kappa$ has a limit in $\fK$. 
A \emph{$\delta$-sequence} in $\fK$ is an inverse sequence (a contravariant functor) whose domain is the ordinal $\delta$ viewed as a linearly ordered category. A $\delta$-sequence $\vec x$ is \emph{continuous} if for every limit ordinal $\eta<\delta$ the $\eta$th object $X_\eta$ is the limit of $\vec x \rest \eta$. Given a sequence $\vec x$, we denote by $x_\alpha^\beta$ the bonding arrow from $X_\beta$ to $X_\alpha$ ($\alpha \loe \beta$).

We say that $\fK$ is \emph{directed}
if for every $A,B\in\obj{\fK}$ there are $C\in\obj{\fK}$ and arrows $f\in\fK(C,A)$ and $g\in\fK(C,B)$.

A family of arrows $\mathcal{F}$ is {\em dominating} in $\mathfrak{K}$ if it satisfies the following two conditions.
\begin{enumerate}\label{domin}
\item[(i)] For every $X\in\obj\fK$ there is 
$A\in\Cod(\mathcal{F})$ such that $\mathfrak{K}(A, X)\ne\emptyset$.
\item[(ii)]Given $A\in\Cod(\mathcal{F})$ and $f\in \fK(Y,A)$  there exist $g\in\fK(B,Y)$  such that $f\circ g\in\mathcal{F}$.
\end{enumerate}
Here, $\Cod(\mathcal{F})$ is the collection of all codomains of the arrows from $\mathcal{F}$.

A {\em $\kappa$-\fra sequence} in $\mathfrak{K}$ is an inverse sequence $\vec{u}$ of length $\kappa$ satisfying the following conditions:
\begin{enumerate}
\item[(U)] For every object $X$ of $\mathfrak{K}$ there exists $\alpha<\kappa$ such that $\mathfrak{K}(U_\alpha, X)\ne\emptyset$.
\item[(A)] For every $\alpha<\kappa$ and for every morphism $f\colon Y\to U_\alpha$, where $Y\in \Ob(\mathfrak{K})$, there exist $\beta\ge\alpha$ and $g\colon U_\beta\to Y$ such that $u_\alpha^\beta=f\circ g$.
 \end{enumerate}
 We say that $\fK$ has the \emph{amalgamation property} if for every $A,B,C\in\obj{\fK}$ and for every morphisms $f\in\fK(B,A)$, $g\in\fK(C,A)$ there exist $D\in\obj{\fK}$ and morphisms $f'\in\fK(D,B)$ and $g'\in\fK(D,C)$ such that $f\circ f' = g\circ g'$.

\begin{theorem}[\cite{KubFra} Theorem~3.7]\label{exists}
Let $\kappa$ be an infinite regular cardinal and let $\mathfrak{K}$ be a $\kappa$-complete directed category with the amalgamation property. Assume further that $\mathcal{F}\subs \mathfrak{K}$  is dominating in $\mathfrak{K}$ and $\vert\mathcal{F}\vert\le\kappa$. Then there exists a continuous \fra  sequence of length $\kappa$ in $\mathfrak{K}$.
\end{theorem}

In fact,~\cite[Theorem~3.7]{KubFra} assumes that the category $\fK$ is $\kappa$-bounded, which is weaker than being $\kappa$-complete, however, an obvious adaptation of the proof gives a continuous \fra sequence whenever $\fK$ is $\kappa$-complete.

\begin{theorem}[\cite{KubFra} Theorem~3.12]\label{uni}
Let $\mathfrak{K}$ be a category with the amalgamation property and let $\vec{u}$ be a \fra sequence of regular length $\kappa$ in $\mathfrak{K}$. Then for every continuous sequence $\vec{x}$ in  $\mathfrak{K}$ of length $\le\kappa$ there exists an arrow of sequences $F\colon \vec{u}\to\vec{x}$.
\end{theorem}

Let us assume that $\mathfrak{K}$ is a full subcategory of a bigger category $\mathfrak{L}$ such that
the following compatibility conditions are satisfied.
\begin{enumerate}
\item [(L0)] All $\mathfrak{L}$-arrows are epi.
\item [(L1)] Every inverse sequence of length $\kappa$ in $\mathfrak{K}$ has the limit in 
$\mathfrak{L}$.
\item [(L2)] Every $\mathfrak{L}$-object is the limit of an inverse sequence in $\mathfrak{K}$.
\item
[(L3)] For every inverse sequence $\vec{x}=\{X_\alpha\colon \alpha<\kappa\}$ in $\mathfrak{K}$ with $X = \lim \vec{x}$ in $\mathfrak{L}$, for every $\mathfrak{K}$-object $Y$,
for every $\mathfrak{L}$-arrow $f\colon X \to Y$ there exist $\alpha<\kappa$ and a $\mathfrak{K}$-arrow $f'\colon  X_\alpha \to Y$ such
that $f = f'\circ  x_\alpha,$ where $x_\alpha\colon X\to X_\alpha$ is the projection.
\end{enumerate}

Now, an $\mathfrak{L}$-object $U$ will be called {\em $\mathfrak{K}$-generic} if
\begin{enumerate}
\item [(G1)] $\mathfrak{L}(U,X) \ne\emptyset$ for every $X \in \text{Obj} (\mathfrak{K})$.
\item
[(G2)] For every $\mathfrak{K}$-arrow $f\colon Y \to X$, for every $\mathfrak{L}$-arrow $g\colon U \to X$ there exists an
$\mathfrak{L}$-arrow $h\colon U \to Y$ such that $f \circ h = g$.
\end{enumerate}

For example, consider a category $\mathfrak{Fin}$ consisting of finite nonempty discrete spaces as a full subcategory of the category $\mathfrak{Comp}$ of compact metric spaces and continuous surjections. It is rather obvious that the Cantor set is the $\mathfrak{Fin}$-generic object.

Another example is obtained if we consider a category $\fL$ consisting of compact spaces of weight not greater than the continuum, with continuous surjections and its subcategory $\mathfrak{Comp}$.
The following results of Parovi\v{c}enko and Negrepontis imply that $\omega^*$ is {$\mathfrak{Comp}$-generic}.

\begin{itemize}
	\item Every compact metric space  is a continuous image of $\omega^*$ (see \cite{par}, \cite{neg}).
\item Assuming [CH], $\omega^*$ is compact Hausdorff space  of weight $\omega_1$ such that for every two continuous surjections $f\colon \omega^*\to X$ and $g\colon Y\to X$ with $X$ and $Y$  compact metric, there  exists a continuous  surjection $h\colon \omega^*\to Y$ such that $g\circ h=f$ (see \cite{neg}).
\end{itemize}

In \cite{BKW} the authors have considered a category
such that $K$ is fixed compact $0$-dimensional second countable topological space and the objects  are continuous mappings $f\colon  K\to X$, from $K$ to a compact $0$-dimensional metric space and a morphism  between two objects $f_0\colon  K\to X_0$ 
and $f_1\colon K\to X_1$ is
a continuous measure preserving surjection $q \colon  X_1\to X_0$ satisfying $q\circ f_1 = f_0.$ They
showed that an embedding $\eta \colon K\to 2^\omega$ such that the image $\eta[K]$ is of measure zero in $2^\omega$ is generic object in the defined category.

The concept of a generic object is strictly related to \fra sequences.

\begin{theorem}\label{fra-gen}
Assume that  $\mathfrak{K} \subs \mathfrak{L}$ satisfy {\rm(L0)}--{\rm(L3)} and $\mathfrak{K}$ has the amalgamation property. If  $\vec{u}$ is a \fra sequence in $\mathfrak{K}$, then $U=\lim\vec{u}$ is a $\mathfrak{K}$-generic.
\end{theorem}
\begin{pf}
	Assume that $\mathfrak{K}$ is subcategory of category $\mathfrak{L}$ satisfying conditions (L0)--(L3). 
Let $\vec{u}$ be a \fra sequence in $\mathfrak{K}$. By (L1),  $U=\lim\vec{u}$ is a $\mathfrak{L}$-object. Consider $\mathfrak{K}$-arrow $f\colon Y\to X$ and $\mathfrak{L}$-arrow $g\colon U\to X$. Then applying condition (L3) we get $\alpha<\omega_1$ and $\mathfrak{K}$-arrow $g'\colon U_\alpha\to Y$   such that $g'\circ u_\alpha^\infty=g$. Amalgamation property implies that there exist $\mathfrak{K}$-arrows $g''\colon Z\to U_\alpha$ and $f'\colon Z\to Y$ such that $g'\circ g''=f\circ f'$.  By condition (A) of the \fra sequence $\vec{u}$ there are $\beta\ge \alpha$ and $\mathfrak{K}$-arrow $h'\colon U_\beta \to Z$ such that $g''\circ h'=u^\beta_\alpha$. Let $h=f'\circ h'\circ u_\beta^\infty$.  
\labelmargin-{1pt}
\amor
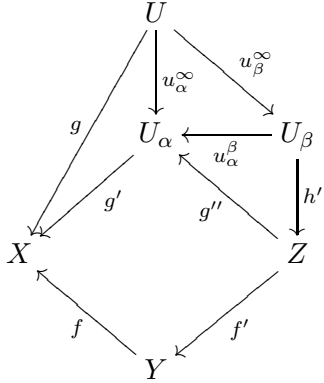
$$\xymatrix@R=2.5pc@C=2.8pc
{&U\ar[d]^(.56){u_\alpha^\infty}\ar[dr]^(.57){u_\beta^\infty}\ar[ldd]_g&\\
&U_\alpha\ar[ld]^(.41){g'}& U_\beta\ar[d]^{h'}\ar[l]^{u_\alpha^\beta}\\
X&&Z\ar[ul]^{g''}\ar[dl]^{f'}\\
&Y\ar[ul]^f &}$$
\end{pf}

\begin{theorem}\label{gen-fra}
	Assume that  $\mathfrak{K} \subs \mathfrak{L}$ satisfy {\rm(L0)}--{\rm(L3)}. 
	If  $U=\lim\vec{u}$ is $\mathfrak{K}$-generic, then $\vec{u}$ is a \fra sequence in $\mathfrak{K}$.
\end{theorem}
\begin{pf}
Condition (G1) combined with (L3) shows that the sequence   $\vec{u}$ satisfies (U). In order to check (A), fix
 an ordinal $\alpha<\kappa$ and a $\mathfrak{K}$-arrow $f\colon Y\to U_\alpha$. By condition (G2) there is $\mathfrak{L}$-arrow
 $g\colon U\to Y$ such that $f\circ g=u^\infty_\alpha$. Using (L3) we get $\beta<\kappa$ and $\mathfrak{K}$-arrow $h\colon U_\beta\to Y$
 such that $g=h\circ u^\infty_\beta$. We can assume that $\beta>\alpha$.
 \amor
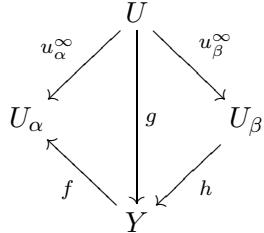
 $$\xymatrix{&U\ar[dr]^{u^\infty_\beta}\ar[dl]_{u^\infty_\alpha}\ar[dd]^{g}&\\
 	U_\alpha && U_\beta\ar[dl]^{h}\\
 	&Y\ar[ul]^{f}&
 }$$
 
Since   $f\circ h\circ u^\infty_\beta=u^\beta_\alpha\circ u^\infty_\beta$ and $ u^\infty_\beta$ is an epimorphism we get $f\circ h=u^\beta_\alpha.$
\end{pf}

\section{Categories of continuous mappings}\label{sekcja:2}

Fix a compact 0-dimensional space $K$ such that $w(K)\le 2^\omega$.  We define the
categories $\mathfrak{L}_K,\mathfrak{C}_K$ as follows.

The objects of $\mathfrak{L}_K$ are continuous mappings $f\colon  K\to X$, where $X$ is compact 0-dimensional with $w(X)\leq \omega_1$.
Given two $\mathfrak{L}_K$-objects $f_0\colon  K \to X_0$, $f_1 \colon  K \to X_1$, an $\mathfrak{L}_K$-arrow from $f_1$ to $f_0$ is
a continuous surjection $q \colon  X_1\to X_0$ satisfying $q \circ  f_1 = f_0$. The composition in $\mathfrak{L}_K$ is the usual composition of mappings.  We define $\mathfrak{C}_K$ to be the full
subcategory of $\mathfrak{L}_K$ whose objects are those $f \colon  K \to X$ with $w(X)\leq\omega$.

\begin{lemma}\label{l:1}
For any space $K$ the category $\mathfrak{C}_K$ 
is directed and has the amalgamation property.
\end{lemma}
\begin{pf}
Let $f\colon K\to X$,  $g\colon K\to Y$, $h\colon K\to Z$ be  $\mathfrak{C}_K$-objects and $q_1\colon f\to h$, $q_2\colon g\to h$ be $\mathfrak{C}_K$-arrows. By the definition $q_1\colon X\to Z$, $q_2\colon Y\to Z$ and $q_1\circ f=h=q_2\circ g$. Consider pullback of  $q_1$ and $q_2$, i.e. the space $W=\{(x,y)\in X\times Y\colon q_1(x)=q_2(y)\}$ together with the pair of projections $f_1\colon W\to X$ and $g_1\colon W\to Y$. Specifically, $f_1(x,y)=x$, $g_1(x,y)=y$. Then $q_1\circ f_1=q_2\circ g_1$.
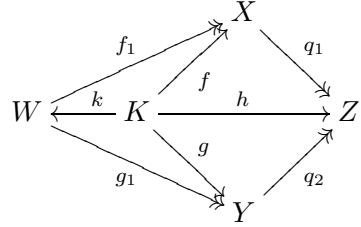
\SelectTips{cm}{11}
$$\xymatrix{
		& &X\ar@{->>}[dr]^{q_1} &\\
		W\ar@{->>}[rru]^{f_1}\ar@{->>}[drr]_{g_1} & K\ar[ur]_f\ar[rd]^g\ar[rr] ^h\ar[l]_(.36){k}&& Z\\
		&& Y\ar@{->>}[ur]_{q_2} & }$$
{From} the pullback property it follows that there exists a unique $k\colon K\to W$ such that  $f=f_1\circ k$ and $g=g_1\circ k$.  Therefore, $f_1$ is $\mathfrak{C}_k$-arrow from $k$ to $f$ and $g_1$ is $\mathfrak{C}_k$-arrow from $k$ to $g$  and the following diagram
	 $$
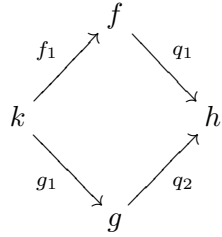
\xymatrix{&f\ar[dr]^{q_1}&\\
	 k \ar[ur]^{f_1} \ar[dr]_{g_1}&&h\\
	 &g\ar[ur]_{q_2}&
	 }$$
	 commutes. This means that the category $\mathfrak{C}_K$ has the  amalgamation property.
	 
	 In order to show that $\mathfrak{C}_K$ is directed, let $f\colon K\to X$ and 
	$g\colon K\to Y$ be continuous mappings from $K$ to  metric spaces $X,Y$, respectively.
Set $Z=X\times Y$ and $h=(f,g)$. The map $h\colon K\to Z$ is continuous and the projections $\pi_X\colon Z\to X$, $\pi_Y\colon Z\to Y$ are continuous surjections such that $\pi_X\circ h=f$ and $\pi_Y\circ h=g$, which completes the proof that $\mathfrak{C}_K$ is directed.
\end{pf}

\begin{theorem}\label{th:5}
Assuming CH, there exists a continuous $\omega_1$-\fra sequence in $\mathfrak{C}_K.$
\end{theorem}
\begin{pf}
Obviously, $\mathfrak{C}_K$	is  $\omega_1$-complete. Since we assume CH, there are $\omega_1$ continuous maps from the compact $0$-dimensional  space $K$ to a compact 0-dimensional metric space. By Theorem \ref{exists} and Lemma \ref{l:1} there  exists a continuous $\omega_1$-\fra sequence in $\mathfrak{C}_K$.
\end{pf}

\nic{The following two lemmas seem to be well known.

\begin{lemma}\label{l:3}
	If $f\colon X\to Y$ is an irreducible map then
	\begin{enumerate}
		\item $f[G]$ is regular closed for each regular closed set $G$;
		\item If $G_0, G_1$ are regular closed sets with $\int G_0\cap\int G_1=\emptyset$ then $\int f[G_0]\cap\int f[G_1]=\emptyset;$
		\item if  $G_0, G_1$ are regular closed sets and $G_0\cup G_1=X$ then $\int f[G_0]\cup\int f[G_1]$ is dense set.
	\end{enumerate}
\end{lemma}

\begin{lemma}\label{l:21}
	If $G_0, G_1$ are regular closed subsets of $X$ such that $\int G_0\cup\int G_1$ is dense set and $\int G_0\cap\int G_1=\emptyset$, then
	$\int G_1=X\setminus G_0$ and $\int G_0=X\setminus G_1$ and the projection $h\colon G_0\times\{0\}\oplus G_1\times\{1\}\to X$ is continuous.
\end{lemma}}	
A compact space is an $F$\textit{-space} if disjoint open $F_\sigma$-subsets have disjoint closures.

\begin{theorem}[\cite{Nev-Lloyd}]\label{project}
	If $X,Y,Z$ are compact 0-dimensional spaces and $f\colon Y\to Z$ is a surjection and $h\colon X\to Z$ is a map and $Y$ is a metrizable space and $X$ is a F-space, then there is a map $g\colon X\to Y$ such that $f\circ g=h.$
	\amor
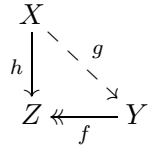
	$$\xymatrix{
		X\ar[d]_h \ar@{-->}[dr]^g &\\
		Z & Y\ar@{->>}[l]^f}$$
		
\end{theorem}

A subset $P$ of a topological space $X$ is called a \textit{P-set} if the intersection of countably many neighborhoods of $P$ is a neighborhood of $P$.

\begin{theorem}\label{fra-sequ}
Assume that  $\vec{\phi}=(\phi_\alpha\colon \alpha<\omega_1)$  is a continuous \fra $\omega_1$-sequence in $\mathfrak{C}_K$, where $\phi_\alpha\colon K\to U_\alpha$ for each $\alpha<\omega_1.$  Then 
\begin{enumerate}
\item[{\rm(1)}] $\vec{u}=(U_\alpha\colon \alpha<\omega_1)$ is a {\fra}sequence in the category $\mathfrak{Comp}$, whenever $K$ is F-space.
\item[{\rm(2)}]If $\w(K)\leq \omega_1,$ then the limit map $\phi_{\omega_1}\colon K\to \lim \vec{u}$ has a left inverse i.e. there is $r\colon \lim \vec{u}\to K$
such that $r\circ \phi_{\omega_1}=\ide_K$,
\item[{\rm(3)}] The image $\phi_{\omega_1}[K]$ is a nowhere dense P-set in $U_{\omega_1}=\lim \vec{u}$. 
\end{enumerate}

\end{theorem}

\begin{pf}(1)
It is immediate that the sequence $\vec{\phi}$ induces $\omega_1$-sequence $\vec{u}=(U_\alpha\colon\alpha<\omega_1)$ in the category $\mathfrak{Comp}$.  
 Let $Y$ be a compact metric space and $p\colon Y\to U_\alpha$ be a continuous surjective map. By Theorem \ref{project} there is a continuous map $g\colon K\to Y$ such that $p\circ g=\phi_\alpha$. In other words, in the category $\mathfrak{C}_K$, $f$ is a morphism from $g$ to $\phi_\alpha$. So, there exists $\beta\ge \alpha$ and a morphism $h\colon \phi_\beta\to g$ such that $p\circ h=u_\alpha^\beta$ because $\vec{\phi}$ is a {\fra}sequence in $\mathfrak{C}_K$.

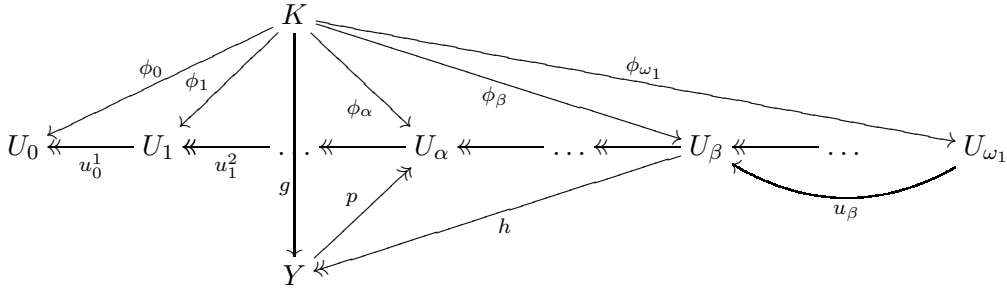
\amor
\labelmargin-{1.8pt}
$$\xymatrix@R=2.8pc@C=2.7pc
{&&K\ar[dll]_{\phi_0}\ar[dl]_(.6){\phi_1}\ar[dr]_(0.6){\phi_\alpha}\ar[drrr]_{\phi_\beta}\ar[dd]_(.66)g\ar[drrrrr]^
	{\phi_{\omega_1}}&&&\\ 
\labelmargin+{1pt}
U_0& U_1\ar@{->>}[l]^{u^1_0}&\dots\ar@{->>}[l]^{u^2_1}\ar@{->>}[l]& U_\alpha\ar@{->>}[l]&\dots\ar@{->>}[l]& U_\beta\ar@{->>}[dlll]^h\ar@{->>}[l]&\dots\ar@{->>}[l]&U_{\omega_1}\ar@/^1.6pc/[ll]^{u_\beta}\\
&& Y\ar@{->>}[ru]^p&&&
}$$

(2) We can assume that $K$ is the limit of a sequence  $\vec{k}=(K_\alpha\colon \alpha<\omega_1)$  in $\mathfrak{Comp}$ 
such that $\ide_K=\lim (p_\alpha\colon \alpha<\omega_1)$ where $p_\alpha$ denotes projections from $K$ onto $K_\alpha$. By Theorem~\ref{uni} there
exists an arrow of sequences $F\colon \vec{\phi}\to (p_\alpha\colon \alpha<\omega_1)$. For the limit map $r=\lim F\colon U_{\omega_1}\to K$ we have 
$r\circ \phi_{\omega_1}=\ide_K$. 

(3) Suppose that there exists a clopen non-empty set $W\subs U_\alpha$ such that  $u_\alpha^{-1}(W)\subs \phi_{\omega_1}[K]$, where $u_\alpha$ is the projection from the limit $U_{\omega_1}$ onto the space $U_\alpha$. Let $Y=(U_\alpha\times\{0\})\oplus (W\times\{1\})$ and define a map $g\colon K\to Y$ by setting $g(x) = (\phi_\alpha(x), 0)$ for every $x\in K.$  
Since the projection 
$p\colon Y\to U_\alpha$ is a morphism from $g$ to $\phi_\alpha$ and $\vec{\phi}$ is a \fra sequence in $\mathfrak{C}_K$ there are $\beta\ge\alpha$ and $h\colon U_\beta\to Y$ such that $p \circ h=u_\alpha^\beta$ and $h\circ \phi_\beta=g.$ Note that 
\begin{multline*}
\emptyset\ne u_\beta^{-1}(h^{-1}(W\times\{1\}))\subs u_\beta^{-1}(h^{-1}(p^{-1}(W)))=
u_\beta^{-1}((p \circ h)^{-1}(W))=\\=
u_\beta^{-1}(( u_\alpha^\beta)^{-1}(W))=u_\alpha^{-1}(W) \subs \phi_{\omega_1}[K].
\end{multline*}
Therefore
$$\phi_{\omega_1}^{-1}(u_\beta^{-1}(h^{-1}(W\times\{1\}))\cap K\ne\emptyset.$$
But $h\circ u_\beta\circ \phi_{\omega_1}=g$ and $g^{-1}(W\times\{1\})=\emptyset$, a contradiction. 

Let $\{W_n\colon n\in\omega\}$ be a family of clopen subsets of the space $U_{\omega_1}$ such that
 $$\phi_{\omega_1}[K]\subs\bigcap\{W_n\colon n\in\omega\}.$$ 
 There exist $\alpha<\omega_1$ and a family 
 $\{V_n\colon n\in\omega\}$  of clopen subsets of  $U_\alpha$ such that $W_n=u_\alpha^{-1}(V_n)$ 
 for every $n\in\omega$. Consider the closed subset  $V=\bigcap\{V_n\colon n\in\omega\}$ of $U_\alpha$. 
 Let $Y=U_\alpha\times\{0\}\oplus V\times\{1\}$  and define a map $g\colon K\to Y$ as follows 
 $g (x) = (\phi_\alpha(x),1)$ for every $x\in K.$  Since $\phi_{\omega_1}(x)\in u_\alpha^{-1} (V)$ 
 for each $x\in K$ the map $g$ is well defined. There exist $\beta\ge \alpha$ and $h\colon U_\beta\to  Y$ 
 such that $p \circ h=u_\alpha^\beta$. Obviously, the set $u_\beta^{-1}(h^{-1}(V\times\{1\}))$ is open in $U_{\omega_1}$. 
 Since $g=h\circ u_\beta\circ \phi_{\omega_1}$ we get $\phi_{\omega_1}[K]\subs u_\beta^{-1}(h^{-1}(V\times\{1\}))$. 
 It remains to check that $$ u_\beta^{-1}(h^{-1}(V\times\{1\}))\subs \bigcap\{W_n\colon n\in\omega\}.$$
We have the following
$$u_\beta^{-1}(h^{-1}(V\times\{1\}))\subs u_\beta^{-1}(h^{-1}(p^{-1}(V)))=u_\alpha^{-1}(V)=\bigcap
\{W_n\colon n\in\omega\},$$
which completes the proof. 
 \end{pf}  
\begin{corollary}[\cite{Bal-Fra-Mills}] Assuming CH, every compact 0-dimensional F-space of weight $\mathfrak{c}$ can be embedded as a nowhere dense closed P-set in $\omega^*.$
\end{corollary}
\begin{pf}
	Let $K$ be a compact 0-dimensional F-space of weight $\mathfrak{c}$. By Theorem \ref{th:5} there exists a continuous $\omega_1$-\fra sequence 
	$(\phi_\alpha:\alpha<\omega_1)$ in $\mathfrak{C}_K,$ where $\phi_\alpha\colon K\to U_\alpha$ for each $\alpha<\omega_1.$ Theorem \ref{fra-sequ}(1) implies that $\vec{u}=(U_\alpha\colon \alpha<\omega_1)$ is a {\fra}sequence in the category $\mathfrak{Comp}$.
	By Theorem \ref{fra-gen}, the limit $U_{\omega_1}=\lim{\vec{u}}$ is $\mathfrak{Comp}$-generic object. Therefore by a characterization of Negrepontis \cite{neg}, $U_{\omega_1}$ is homeomorphic to $\omega^*.$ Using Theorem \ref{fra-sequ}(2), the map
	$\phi_{\omega_1}\colon K\to\lim{\vec{u}}$ is an embedding and by  Theorem \ref{fra-sequ}(3) the image $\phi_{\omega_1}[K]$ is a
	nowhere dense closed P-set in $\omega^*.$
\end{pf}

\begin{theorem}\label{generic}
Let $\eta\colon K\to\omega^*$ be an embedding such that $\eta[K]$ is a nowhere dense P-set  of $\omega^*$ and $K$ is a
0-dimensional compact space. Then $\eta\colon K\to\omega^*$ is $\mathfrak{C}_K$-generic.
\end{theorem}

\begin{pf}
	It  suffices to show that $\eta$ satisfies condition (G2), i.e. given 0-dimensional compact metrizable spaces $X,Y$ and a surjection $f\colon X\to Y,$
	and a continuous map $b\colon K\to Y,$ and a continuous surjection $g\colon \omega^*\to X$ such that $g\circ\eta=f\circ b$, there exists a continuous surjection $h\colon \omega^*\to Y$ such that
	\[f\circ h=g \text{ and } h\circ \eta=b.\] 
	\amor
	$$\xymatrix@R=2.5pc@C=2.9pc{
		K\ar[dr] \ar[d]_{b}\ar@{^(->}[r]^{\eta}&\omega^*\ar@{..>>}[dl]_(.34){h}\ar@{->>}[d]^{g}\\
		Y\ar@{->>}[r]_{f}&X}$$
	Since $X, Y$ are 0-dimensional compact metrizable spaces we can represent the spaces $X,Y$ as an inverse limits  $X=\liminv\{X_n,g_n^{n+1},n\in\omega\}$ and $Y=\liminv\{Y_{n},b_n^{n+1},n\in\omega\}$, respectively, and  such that $X_n, Y_{n}$ 
	are finite discrete spaces and the following diagram commutes
	
	\amor
	$$\xymatrix@R=2.5pc@C=2.9pc{
		Y\ar@{->>}[r]^{f}\ar@{->>}[d]_{b_{n+1}}&X\ar@{->>}[d]^{g_{n+1}}\\
		Y_{{n+1}} \ar@{->>}[d]_{b^{n+1}_n}\ar@{->>}[r]^{f_{n+1}}&X_{n+1}\ar@{->>}[d]^{g^{n+1}_n}\\
		Y_{n}\ar@{->>}[r]_{f_n}&X_n}$$
	for each $n\in \omega.$	
	It's sufficient to construct for each $n\in\omega$ a continuous surjection $h_n\colon \omega^*\to Y_{n}$ such that
	$$g_n\circ g=f_n\circ h_n\mbox{ and }h_n\circ\eta=b_n\circ b$$
	\amor
	$$\xymatrix@R=2.5pc@C=2.9pc{
		K \ar[d]_{b}\ar@{^(->}[r]^{\eta}&\omega^*\ar@{..>>}[ddl]_(.24){h_n}\ar@{->>}[d]^{g}\\
		Y\ar@{->>}[d]_{b_n}\ar@{->>}[r]^(.3){f}&X\ar@{->>}[d]^{g_n}\\
		Y_{n}\ar@{->>}[r]_{f_n}&X_n }$$
Then let $h$ be the map induced by $\{h_n:n\in\omega\}.$
	
	\textit{The construction of the map $h_{n}$.}
	
	Since the above diagrams are commutative for each $y\in Y_n$ and $n\in\omega$ we have 
	\[(*)\;f(b_n^{-1}(y))\subs g_n^{-1}(f_n(y))\mbox{ and }\eta(b^{-1}(b_n^{-1}(y)))\subs g^{-1}(g_n^{-1}(f_n(y))).\]
	For each $y\in Y_n$ and $n\in\omega$ we choose a countable  subset $A^n_y\subs (b_n)^{-1}(y)$ such that 
	$\cl A^n_y=(b_n)^{-1}(y)$ and we put
	\[A_y=\bigcup\{A^k_z:z\in Y_k, \;b^k_n(z)=y,\; n\leq k<\omega\}\subs Y.\]
	Therefore we get the following:	
	
	\begin{enumerate}[(a)]
		\item if $s, t\in Y_n$ and $s\ne t$ then $A_s\cap A_t=\emptyset$,
		\item  if $s\in Y_n$ and $t\in Y_k$ and $n<k$ and $b^k_n(t)=s$ then $A_t\subs A_s$,
		\item $\bigcup\{f[A_y]:y\in f_n^{-1}(x)\}$ is a dense subset of $g_n^{-1}(x)$ for $x\in X_n$.
	\end{enumerate}

	Inductively we will define a collection $\{\Bee_y:y\in\bigcup\{Y_n:n\in\omega\}\}$ of countable families consisting of clopen sets of $\omega^*$ disjoint from $\eta[K]$  such that
	\begin{enumerate}[(i)]
		\item if $s, t\in Y_n$ and $s\ne t$ then $\bigcup \Bee_s\cap\bigcup\Bee_t=\emptyset$,
		\item if $s\in Y_n$ and $t\in Y_k$ and $n<k$ and $b^k_n(t)=s$ then $\bigcup\Bee_t\subs \bigcup\Bee_s$,
		\item for each $x\in f[A_y]$ there is $V\in\Bee_y$ such that $V\subs\int g^{-1}(x)$ and  each $V\in\Bee_y$
		is included in the set $\int g^{-1}(x)$ for some $x\in f[A_y].$
	\end{enumerate}
	
	Assume that we have just defined $\{\Bee_y:y\in\bigcup\{Y_k:k\leq n\}\}$ with properties $(a)-(c)$. 
	Since the family $\{(b^{n+1}_n)^{-1}(y)\cap f^{-1}_{n+1}(x):x\in X_{n+1}, y\in Y_n\}$ consists of pairwise disjoint sets and $Y_{n+1}=\bigcup\{(b^{n+1}_n)^{-1}(y)\cap f^{-1}_{n+1}(x):x\in X_{n+1}, y\in Y_n\}$ given $x\in X_{n+1}$ and $y\in Y_n$ we define a family $\Bee_t$ for $t\in (b^{n+1}_n)^{-1}(y)\cap f^{-1}_{n+1}(x).$
	Let $(b^{n+1}_n)^{-1}(y)\cap f^{-1}_{n+1}(x)=\{y_0,\ldots ,y_k\}$ and 
	$a\in f[A_{y_0}]\cup\ldots\cup f[A_{y_k}]\subs g^{-1}_{n+1}(x)\cap f[A_y].$ There is the unique $V\in \Bee_y$ such that
	$V\subs \int g^{-1}(a)$. Let $\{i_0,\ldots ,i_p\}=\{j:a\in f[A_{y_j}], 0\leq j\leq k\}.$ 
	Then we divide the set $V$ onto finite clopen sets
	$V=V_{i_0}\cup\ldots\cup V_{i_p}$ and we put 
	$$\Bee_{y_j}=\{V_j:V\in\Bee_y, a\in  f[A_{y_j}], V\subs \int g^{-1}(a)\}$$ 
	for $0\leq j\leq k.$ 
	
	We shall construct a partition $$\Pee_n=\{W_y:y\in Y_n,\eta(b^{-1}(b_n^{-1}(y)))\ne\emptyset \}\cup\{V_y:y\in Y_n\}$$  of $\omega^*$ consisting of clopen sets such that
	\begin{enumerate}
		\item $\eta(b^{-1}(b_n^{-1}(y)))\subs W_y$ whenever $\eta(b^{-1}(b_n^{-1}(y)))\ne\emptyset$ and $y\in Y_n,$
		\item $\bigcup\Bee_y\subs V_y$ for $y\in Y_n,$
		\item if $y\in Y_n$ then 
		\begin{itemize}
			\item $V_y\subs g^{-1}((g_n)^{-1}(f_n(y)))$ whenever $\eta(b^{-1}(b_n^{-1}(y)))=\emptyset$,
			\item $W_y\cup V_y\subs g^{-1}((g_n)^{-1}(f_n(y)))$ whenever $\eta(b^{-1}(b_n^{-1}(y)))\ne\emptyset$,
		\end{itemize}
		
		\item  if $s\in Y_n$ and $t\in Y_k$ and $n<k$ and $b^k_n(t)=s$ then 
		\begin{itemize}
			\item $V_t\subs V_s$,
			\item $W_t\subs W_s$ whenever $\eta(b^{-1}(b_k^{-1}(t)))\ne\emptyset$.
		\end{itemize}
	\end{enumerate}
	Assume that we have just defined a partition $\Pee_n$ with properties $(1)-(4)$.  Since $\omega^*$ is an F-space and for 
	each $y\in Y_n$ the set $\bigcup\Bee_y$ is an $F_\sigma$ set disjoint from the P-set $\eta[K]$, the family $\{\cl \bigcup\Bee_y:y\in Y_{n+1}\}$ consists of pairwise disjoint closed sets and disjoint from the set $\eta[K]$. By $(iii)$ and $(*)$ we get 
	$$\cl \bigcup\Bee_y\subs\cl g^{-1}(f(A_y))\subs \cl g^{-1}(g_{n+1}^{-1}(f_{n+1}(y)))=g^{-1}(g_{n+1}^{-1}(f_{n+1}(y)))$$
	for all $y\in Y_{n+1}$. By $(ii)$ and $(1),(2)$ and $(*)$ the family 
	$$\Raa_{n+1}=\{\cl \bigcup\Bee_y:y\in Y_{n+1}\}\cup\{\eta(b^{-1}(b_{n+1}^{-1}(y))):y\in Y_{n+1}\}$$ consists of pairwise disjoint closed sets such that each element 
	of $\Raa_{n+1}$ is included in some element of the partition $\Pee_n$ and in some element of the partition  $\{g^{-1}(g_{n+1}^{-1}(x)):x\in X_{n+1}\}.$ Therefore we find a partition 
	$$\Pee_{n+1}=\{W_y:y\in Y_n,\eta(b^{-1}(b_n^{-1}(y)))\ne\emptyset \}\cup\{V_y:y\in Y_n\}$$  of $\omega^*$ consisting of clopen sets with properties $(1)-(4)$.
	
	Define the map $h_n\colon \omega^*\to Y_n$ in the following way:
	\[h_n(x)=y \mbox{ whenever } x\in W_y\cup V_y. \]
	
	If $x\in K$ there exists a unique $y\in Y_n$ such that $x\in b^{-1}(b_n^{-1}(y))$. 
	Therefore $\eta(x)\in\eta(b^{-1}(b_n^{-1}(y)))$ and by $(1)$ we get $h_n(\eta(x))=y$ and $b(b_n(x))=y.$
	Since $\Pee_n$ is a partition of $\omega^*$, for each $x\in\omega^*$ there exists a unique $y\in Y_n$ such that $x\in W_y\cup V_y$. 
	Hence $h_n(x)=y$. By $(3)$ we get $g(g_n(x))=f_n(y)$. Therefore $g(g_n(x))=f_n(h_n(x)).$ According to $(2)$ the map $h_n$ is a surjection.
\end{pf}

\begin{corollary}[\cite{Mill-Douw}]  Assuming CH, every homeomorphism of nowhere dense closed $P$-sets in  $\omega^*$ 
	can be extended to an auto-homeomorphism of  $\omega^*$.
\end{corollary}
\begin{pf}
Let $h\colon K\to L$ be a homeomorphism between nowhere dense closed $P$-sets in  $\omega^*$. Denote by $\eta\colon K\hookrightarrow \omega^*$ 
and $\eta_1\colon L\hookrightarrow \omega^*$ the inclusion mappings. By Theorem \ref{generic}, both $\eta$ and $\eta_1\circ h$ are $\mathfrak{C}_K$-generic, therefore by uniqueness (\cite[Theorem 4.6]{KubFra}) there
exists an $\mathfrak{L}_K$-isomorphism $H \colon  \omega^*\to\omega^*$ from $\eta$ to $\eta_1\circ h$. This means that there exists an auto-homeomorphism of $\omega^*$ extending $h$.
\end{pf}
\begin{corollary}[\cite{Mill-Douw}]  Assuming CH, every  nowhere dense closed $P$-set in  $\omega^*$   is a retract of 
	$\omega^*$.
\end{corollary}
\begin{pf}
	Let $\eta\colon K\hookrightarrow \omega^*$ denote the inclusion mapping of a nowhere dense closed $P$-set $K\subs \omega^*.$ 
	We can represent $\eta\colon  K\hookrightarrow \omega^*$ as the limit of a sequence $\vec{u}=\{\phi_\alpha:\alpha<\omega_1\}$ in the category $\mathfrak{C}_K$, where $\phi_\alpha\colon K\to U_\alpha.$ By Theorem \ref{generic} $\eta\colon  K\hookrightarrow \omega^*$ is $\mathfrak{C}_K$-generic, therefore by Theorem \ref{gen-fra} $\vec{u}$ is \fra sequence.  According to Theorem \ref{fra-sequ}(2) there is a	retraction $r\colon \omega^*\to K$. 
\end{pf}

\section{Categories of $\kappa$-ultrametric  spaces}\label{sekcja:3}

Let $\gamma$ be an ordinal.  A $\gamma$-\textit{ultrametric} (also called an “inverse $\gamma$-metric”, see e.g.~Nyikos \cite{Nyi}) on a set $X$
 is a function  $u\colon X \times X\to \gamma+1$  such that for all
$x, y, z \in X$:
\begin{itemize}
	\item [(U1)] $u(x, y) =\gamma$ if and only if $x = y,$
	\item [(U2)] $u(y,z)\geq\min\{u(y,x),u(x,z)\}$ (ultrametric triangle law),
	\item [(U3)] $u(x, y) = u(y, x)$ (symmetry).
\end{itemize}
 A \textit{closed  ball of  radius $\alpha$ and center $x$}  is a set of the form 
$$B_\alpha(x) =\{y\in X:u(x,y)\geq\alpha\},$$ 
where $x\in X$ and $\alpha\in \gamma$. 
Each $\gamma$-ultrametric induces a topology whose base is the family of all closed balls. We will call a space with this topology  {\em $\gamma$-ultrametric} (or {\em $\gamma$-metrizable}).
It is rather obvious that two balls in an  $\gamma$-ultrametric are either disjoint or one is contained in the other.
 Moreover, if $B_\alpha(a)\subsetneq B_\beta(b)$ then $\alpha>\beta.$ 

If $X$ is a discrete space,  then for any ordinal $\gamma\geq\omega$ there is a natural   

 	$\gamma$-ultrametric 	$d\colon  X\times X\to \gamma+1$ on $X$, namely $d(a,b)=\gamma$ iff $a=b$ and $d(a,b)=0$ iff $a\ne b$.  Since  $B_\gamma(x)=\{x\}$ the set $\{x\}$ is open for any $x\in X$. So in this case  $\gamma$-ultrametric topology is   discrete.

Let $\kappa,\lambda$ be  infinite cardinals. 
We  define a $\lambda$-ultrametric   $u\colon \kappa^\lambda\times \kappa^\lambda\to\lambda+1$ by the formula: $$u(a,b)=\sup\{\alpha < \lambda \colon a\restriction \alpha=b\restriction \alpha\}$$ for $a,b\in\kappa^\lambda$. 
 
	We say that a $\lambda$-ultrametric space $(X,d)$ of weight not greater 
	than $\kappa^{<\lambda}$  is $(\lambda,\kappa)$-\textit{bounded} if there is a non-decreasing sequence of ordinals numbers
	$\{\gamma_\alpha:\alpha<\lambda\}\subset\lambda$ such that $|\{B_\alpha(a)\colon a\in X\}|\leq |\kappa^{\gamma_\alpha}|$ 
	for every $\alpha<\lambda.$
	
	A $\lambda$-ultrametric space $X$ is \textit{spherically complete} if every nonempty chain of closed balls has nonempty intersection.

A space $\kappa^\omega$ with ultrametric defined as above is $(\omega,\kappa)$-bounded.

\begin{theorem}\label{th:4.1}
Assume that  $\kappa,\lambda$ are regular cardinals such that $\lambda\leq\kappa$.	
A topological space $X$ is
$(\lambda,\kappa)$-bounded and spherically complete
if and only if 
there exists a  non-decreasing sequence of ordinals  $\{\gamma_\alpha:\alpha<\lambda\}\subset\lambda$ 
 and exists an inverse sequence  $\{X_\alpha, q_\alpha^\beta,\alpha\le \beta<\lambda\}$ such that
\begin{itemize}
\item $X=\liminv \{X_\alpha, q_\alpha^\beta,\alpha\le \beta<\lambda\}$,
\item $X_\alpha$ is a discrete space of cardinality not greater than  $\vert\kappa^{\gamma_\alpha}\vert$ for each $\alpha<\lambda$,
\item $q_\alpha^\beta\colon X_\beta\to X_\alpha$ is surjection for all $\alpha\le\beta<\lambda$.
\end{itemize}
 \end{theorem}

\begin{proof}
Assume that $X$ is $(\lambda,\kappa)$-bounded and spherically complete. So there is a non-decreasing sequence of ordinals  
	$\{\gamma_\alpha:\alpha<\lambda\}\subset\lambda$   such that 
	$|\{B_\alpha(a)\colon a\in X\}|\leq |\kappa^{\gamma_\alpha}|$	for every $\alpha<\lambda$.  
	Let us consider the family $X_\alpha=\{B_\alpha(a):a\in X\}$ as a discrete space for every $\alpha<\lambda$. For $\alpha\le\beta$ 
	 the formula $q_\alpha^\beta(B_\beta(a))=B_\alpha(a)$ defines surjection  $q_\alpha^\beta\colon X_\beta\to X_\alpha$ . The map  $q_\alpha^\beta$ is well defined because in any ultrametric space any two balls are either disjoint or one is contained in the other. It easy to see that $X=\liminv \{X_\alpha, q_\alpha^\beta,\alpha\le\beta<\lambda\}.$

Let $h\colon X\to\varprojlim\{X_\alpha,q_\alpha^\beta,\alpha\le \beta<\lambda\}$ be a map defined by the formula $h(x)=(B_\alpha(x)\colon \alpha<\lambda)$. Observe that $h$ is injective. Indeed, for distinct $x,y\in X$ we have $d(x,y)=\alpha<\lambda$. So $B_{\alpha+1}(x)\ne B_{\alpha+1}(y)$. 
For the surjectivity  of $h$  fix  $(B_\alpha(x_\alpha)\colon \alpha<\lambda)\in\varprojlim
\{X_\alpha,q_\alpha^\beta,\alpha\le \beta<\lambda\}$. Then $$B_\beta(x_\beta)\subseteq   B_\alpha(x_\beta) =q_\alpha^\beta(B_\beta(x_\beta))=B_\alpha(x_\alpha)$$ for any $\alpha\le \beta<\lambda$ which means that $\mathcal{N}=\{B_\alpha(x_\alpha)\colon \alpha<\lambda\}$ is a chain. Since $X$ is spherically complete then there is $x\in\bigcap\mathcal{N}$.  Therefore $x\in B_\alpha(x_\alpha)$ which means that $B_\alpha(x)=B_\alpha(x_\alpha)$ for each $\alpha<\lambda$ and consequently $(B_\alpha(x_\alpha)\colon \alpha<\lambda)=h(x)$. To prove the continuity of $h$, it is sufficient to note that

$$h^{-1}(q_\alpha^{-1}(B_\alpha(x)))=B_\alpha(x)$$
for each $\alpha<\lambda$. The last equality also ensures that $h$ is an open map, because $h$ is surjective and closed balls form the basis of the space $X$.

	Assume that  $X=\liminv \{X_\alpha, q_\alpha^\beta,\alpha\le \beta<\lambda\},$ where $X_\alpha$ is a discrete space of cardinality not greater than  $|\kappa^{\gamma_\alpha}|$  and $\{\gamma_\alpha:\alpha<\lambda\}\subset\lambda$ is a non-decreasing. Define a $\lambda$-ultrametric   $u\colon X^2\to\lambda+1$ by the formula: $$u(a,b)=\sup\{\alpha < \lambda \colon a\restriction \alpha=b\restriction \alpha\}$$ for $a,b\in X$. 
		Since 
		\begin{equation}
		B_{\alpha+1}(x)=q_\alpha^{-1}(x_\alpha)\tag{$*$}
		\end{equation}
		 for all $x\in X$ and $\alpha<\lambda$
the topology induced by closed balls on $X$ is compatible with topology of inverse limit on $X$.	
	Therefore $X$ is  $(\lambda,\kappa)$-bounded. 
	
	We shall prove that $X$ is spherically complete. Let $\mathcal{N}=\{B_s(a_s)\colon s\in S\}$ be a chain and  $S\subset\lambda$ and $a_s\in X_s.$  Since $B_\alpha(a)\subsetneq B_\beta(b)$ implies $\alpha>\beta$ we can assume that $N=\{B_\alpha(a_\alpha)\colon \alpha<\lambda\}.$ Hence $\bigcap N=\{(a_\alpha)_{\alpha<\lambda}\}.$
	\end{proof}

\begin{theorem}[{\cite[Theorem~2.10]{Nyi}}]\label{ultra}
	Let  $\kappa,\lambda$ be  regular cardinals such that $\lambda\leq \kappa$.  
	
	Then a $\lambda$-ultrametric space $X$  of weight $\kappa$ is homeomorphic to a subspace of the    $\lambda$-ultrametric space $\kappa^\lambda$.
\end{theorem}

\begin{pf}
Let $u\colon X\times X\to \lambda+1$ be a $\lambda$-ultrametric on the set $X.$	Define a discrete space $X_\alpha=\{B_\alpha(x)\colon x\in X\}$ and a surjection $q_\alpha^\beta\colon X_\beta\to X_\alpha$ in the following way $q_\alpha^\beta(B_\beta(x))=B_\alpha(x)$. The map  $q_\alpha^\beta$ is well defined because in an ultrametric space any two balls are either disjoint or one is contained in the other. Obviously $\vert X_\alpha\vert\le \kappa$ for each $\alpha<\lambda$. Let 
$$Y=\{(B_\alpha(x))_{\alpha<\lambda}\colon x\in X\}.$$
It is clear that
$Y\subs \liminv \{X_\alpha, q_\alpha^\beta,\alpha<\beta<\lambda\}$. 
 
The space $X$ is homeomorphic to 
	$Y.$ Indeed, define a map $h\colon X\to Y$ as follows 
	$h(a)=(B_\alpha(a))_{\alpha<\lambda}$. It's easy to see that $h$ is a bijection and $h^{-1}(q_\alpha^{-1}(B_\alpha(a)))=B_\alpha(a)$, hence $h$ is homeomorphism onto a subspace of $\kappa^\lambda$.
\end{pf}

Let $(K,u)$ and $(X,d)$ be  ultrametric spaces, where  $u\colon K\times K\to \kappa+1$ and  $d\colon X\times X\to \tau+1$ and $\kappa,\tau$ are infinite cardinals. 
We say that a function $f\colon K\to X$ is 
\textit{uniformly continuous} if for every $\epsilon\in \tau$ there is $\delta\in \kappa$ such that for all
	$a,b\in K$ if $u(a,b)\geq \delta$, then $d(f(a),f(b))\geq\epsilon$.

\textbf{In the rest of this section we assume that  $\kappa,\lambda$ are regular cardinals such that $\lambda\leq\kappa$.}

Fix a $\lambda$-ultrametric space $(K,u)$ of weight $\kappa.$ 

We define the categories $\mathfrak{M}_K$ and $\mathfrak{D}_K$ as follows. 
The objects of $\mathfrak{M}_K$ are uniformly continuous mappings $f\colon  K\to X$, where $(X,d)$ is a 
 
$(\lambda,\kappa)$-bounded and spherically complete.   Given two $\mathfrak{M}_K$-objects 
$f_0\colon  K \to X_0$, $f_1 \colon  K \to X_1$, an $\mathfrak{M}_K$-arrow from $f_1$ to $f_0$ is
a uniformly continuous surjection $q \colon  X_1\twoheadrightarrow X_0$ satisfying $q \circ f_1 = f_0$.
The composition in $\mathfrak{M}_K$ is the usual composition of mappings.  
We define $\mathfrak{D}_K$ to be the full
subcategory of $\mathfrak{M}_K$ whose objects are those $f \colon  K \to X$ where $X$ is a discrete space of cardinality not greater than $|\kappa^\gamma|$ for some $\gamma<\lambda.$ 

We define the family of arrows $\mathcal{F}_K$   in the following way. 
For each $\alpha,\gamma<\lambda$ we define the discrete space
$$X_\alpha^\gamma = \{B_\alpha(a)\colon a\in K\} \oplus \kappa^\gamma$$
and define $f_\alpha^\gamma \colon K \to X_\alpha^\gamma$, an object from the category $\mathfrak{D}_K$, such that $(f_\alpha^\gamma)^{-1}[B_\alpha(a)] = B_\alpha(a)$ for every $a \in K$.
Given ordinals $\xi<\delta<\lambda$, fix a map $r^\delta_\xi\colon \kappa^\delta\to\kappa^\xi\times\{0,1\}$ such that $|(r^\delta_\xi)^{-1}(\alpha,i)|=|\kappa^\delta|$ for every
$(\alpha,i)\in \kappa^\xi\times\{0,1\}$.
For each $\alpha,\xi<\lambda$ we fix a surjection $p^\xi_\alpha\colon \kappa^\xi\to \{B_\alpha(a)\colon a\in K\}$.
Given $\alpha \loe \beta<\lambda$ and $\xi\leq\delta<\lambda$ we define a  surjection $q^{(\beta,\delta)}_{(\alpha,\xi)} \colon  X^\delta_{\beta}\twoheadrightarrow X^\xi_{\alpha}$ as follows:
\[q^{(\beta,\delta)}_{(\alpha,\xi)}(x)=\begin{cases}
	B_{\alpha}(a),& \text{if }x = B_{\beta}(a),\\
	\pi(r^\delta_\xi(x)),& \text{if }x\in\kappa^\delta\text{ and }r^\delta_\xi(x)\in\kappa^\xi\times\{1\},\\
	p^\xi_{\alpha}(\pi(r^\delta_\xi(x))),& \text{if }x\in\kappa^\delta\text{ and }r^\delta_\xi(x)\in\kappa^\xi\times\{0\},
\end{cases}\]
where $\pi$ is the canonical projection from $\kappa \times \{0,1\}$ onto $\kappa$. Let $$\mathcal{F}_K=\left\{q^{(\beta,\delta)}_{(\alpha,\xi)}:\alpha\loe\beta<\lambda\;\&\; \xi\leq\delta<\lambda\right\}.$$

\begin{lemma}\label{l:2}
	The family $\mathcal{F}_K$ is  dominating in $\mathfrak{D}_K$ and $|\mathcal{F}_K|\loe \lambda.$
\end{lemma}

\begin{proof}
It is obvious that $|\mathcal{F}_K|\loe\lambda.$ We prove that $\mathcal{F}_K$ satisfies condition (i) of ``being dominating''.
Let  $f\colon K\to Z$ be a uniformly continuous, where $Z$ is a discrete space of 
 cardinality not greater than $|\kappa^\gamma|$ for some $\gamma<\lambda$. We can consider the space $Z$ 
 as a disjoint union of the image $f[K]$ and a set $A$, where $|A|\le |\kappa^\gamma|$.
   Since $f$ is  uniformly continuous there is an ordinal  $\alpha<\lambda$  such that for each $x\in f[K]$ 
   there is a family $\mathcal{S}_x\subs\{B_\alpha(a):a\in K\}$ with $f^{-1}(x)=\bigcup \mathcal{S}_x$. 
Let $\beta=\alpha$

 and consider $f_\beta^\gamma\colon K\to X^\gamma_\beta.$ Define an arrow $q$ from $f_\beta^\gamma$ to $f$ as follows 
\[q(x)=
\begin{cases}
	f(y),&\text{if } x=f_\beta^\gamma(y)\text{ and }y\in K,\\
	g(x),&\text{if } x\in\kappa^\gamma,
\end{cases}\]
where $g\colon\kappa^\gamma\to A$ is an arbitrary surjection. 
Since the following diagram 
\amor
	$$\xymatrix{
		K\ar[d]_{f} \ar[dr]^{f_\beta^\gamma} &\\
		Z & X_\beta^\gamma\ar@{->>}[l]^q}$$

commutes then condition (i) is fulfilled.  
	
	In order to prove that $\mathcal{F}_K$ satisfies condition (ii) in the definition of ``being dominated'', assume that we have the commutative diagram
		\amor
	$$\xymatrix{
		&K\ar[ld]_{f_\alpha^\xi} \ar[d]^h &\\
		X_\alpha^\xi & Y\ar@{->>}[l]^f &}$$
		
where $f_\alpha^\xi\colon K\to X_\alpha^\xi$ is an $\mathcal{F}_K$-object 
(codomain of some arrows from $\mathcal{F}_K$), $h$ is a 
$\mathfrak{D}_K$-object and $f$ is $\mathfrak{D}_K$-arrow. The space $Y$ is a disjoint union of $h[K]$ and a set $A$, 
where $|A|\le |\kappa^\delta|$ for some $\delta$ such that $\xi\leq\delta<\lambda$. Let $A_0=f^{-1}(f^\xi_\alpha[K])\cap A$ and 
$A_1=A\setminus A_0$. Since $h$ is  uniformly continuous, there is an ordinal $\beta$ such that  
$\alpha<\beta<\lambda$   and for each $x\in h[K]$ there is a family $\mathcal{S}_x\subs\{B_\beta(a):a\in K\}$ with 
$h^{-1}(x)=\bigcup \mathcal{S}_x$.  
We are going to define a $\mathfrak{D}_K$-arrow $g\colon X_\beta^\delta\to Y$ such that $f\circ g=q^{(\beta,\delta)}_{(\alpha,\xi)}$.	
	$$\begin{tikzcd}
	&	K\arrow[dl,"f_\alpha^\xi"']\arrow[d,"h"]\arrow[dr,"f_\beta^\delta"] &\\
	X_\alpha^\xi&	Y\arrow[two heads]{l}[swap]{f} & X_\beta^\delta\arrow[bend left,two heads]{ll}[below]{q^{(\beta,\delta)}_{(\alpha,\xi)}}\arrow[dashed,two heads]{l}[swap]{g}
	\end{tikzcd}$$
 Since all spaces $X_\alpha^\xi, Y,X_\beta^\delta$ are discrete  we don't have to worry about continuity, we just need to make sure that the fiber goes to the fiber according to the map $q^{(\beta,\delta)}_{(\alpha,\xi)}$.

The situation is visualized in the following diagram.
	$$\begin{tikzcd}
&f_\beta^\delta[K]\arrow[bend left]{dddrr}[near start]{q^{(\beta,\delta)}_{(\alpha,\xi)}}\arrow[dashed]{ddl}{g}&&\kappa^\delta	\arrow{d}{r^\delta_\xi}\arrow{drr}{r^\delta_\xi} &&\\
&&&\kappa^\xi\times\{0\}\arrow[dashed]{dlll}{g}\arrow[dashed]{dl}{g}\arrow[bend left]{dd}{p_\alpha^\xi\circ\pi}
&&\kappa^\xi\times\{1\}\arrow{dd}{\pi}\arrow[dashed]{dl}{g}\\
h[K]\arrow[bend right]{drrr}{f}&&A_0\arrow{dr}[swap]{f}&&A_1\arrow{dr}[swap]{f}&\\
&&&f_\alpha^\xi[K]&&\kappa^\xi
\end{tikzcd}$$

Given $z\in f_\alpha^\xi[K]$,
we put $A^0_z=f^{-1}(z)\cap h[K]$ and $A^1_z=f^{-1}(z)\cap A\subs A_0$, 
whenever 
$A_z^0\ne\emptyset.$ Since $|(r^\delta_\xi)^{-1}[((p_\alpha^\xi)^{-1}(z))\times\{0\}]|=|\kappa^\delta|$, we can 
find a partition $\{B^0_z,B^1_z\}$ of the set $(r^\delta_\xi)^{-1}[((p_\alpha^\xi)^{-1}(z))\times\{0\}]$ such that 
$|B^0_z|=|B^1_z|=|\kappa^\delta|$ and two surjection $h^i_z\colon B^i_z\to A^i_z$, $i\in\{0,1\}.$  In the case $A_z^0=\emptyset$ let 
$h^i_z\colon B^i_z\to A^1_z$ be any surjection for $i\in\{0,1\}.$

Let $z\in\kappa^\xi.$ Since  $|(r^\delta_\xi)^{-1}((z,1))|=|\kappa^\delta|$, there is a surjection $h_z\colon  (r^\delta_\xi)^{-1}((z,1))\to f^{-1}(z).$ 
We define a map $g \colon X_\beta^\delta \to Y$ as follows

\[g(x)=\begin{cases}
	h^i_z(x)&\text{ if }  x\in B^i_z\subs (r^\delta_\xi)^{-1}[((p_\alpha^\xi)^{-1}(z))\times\{0\}]\text{ and } z\in f^\xi_\alpha[K],\\
h_z(x)	&\text{ if }  x\in (r^\delta_\xi)^{-1}((z,1))\text{ and } z\in \kappa^\xi,\\
h(y) &\text{ if } x=f^\delta_\beta(y)\text{ and } y\in K. 
\end{cases}\]

It is easy to check that $h$  has the desired properties.
\end{proof}

We now mimic the considerations of Section~\ref{sekcja:2}, proving the following.

\begin{lemma}\label{l:22}
	The category $\mathfrak{D}_K$  is directed and has the amalgamation property.
\end{lemma}

\begin{pf}
	Let $f\colon K\to X$,  $g\colon K\to Y$, $h\colon K\to Z$ be  $\mathfrak{D}_K$-objects and $q_1\colon f\to h$, $q_2\colon g\to h$ be $\mathfrak{D}_K$-arrows. By the definition $q_1\colon X\to Z$, $q_2\colon Y\to Z$ and $q_1\circ f=h=q_2\circ g$. Consider  pullback of  $q_1$ and $q_2$ i.e. the space $W=\{(x,y)\in X\times Y\colon q_1(x)=q_2(y)\}$  along with a pair of surjections $f_1\colon W\to X$ and $g_1\colon W\to Y$ such that $q_1\circ f_1=q_2\circ g_1$.
	\SelectTips{cm}{11}
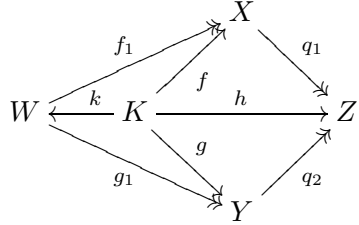
	$$\xymatrix{
		& &X\ar@{->>}[dr]^{q_1} &\\
		W\ar@{->>}[rru]^{f_1}\ar@{->>}[drr]_{g_1} & K\ar[ur]_f\ar[rd]^g\ar[rr] ^h\ar[l]_(.36){k}&& Z\\
		&& Y\ar@{->>}[ur]_{q_2} & }$$ 
	
	From the pullback property it follows that exists a unique $k\colon K\to W$ such that  $f=f_1\circ k$ and $g=g_1\circ k$.  Since $f,g$ are uniformly continuous and $k^{-1}(x,y)=f^{-1}(x)\cap g^{-1}(y)$ for $(x,y)\in W$ and balls in an ultrametric space are either disjoint or comparable, the map $k$ is uniformly continuous. Therefore, $f_1$ is a $\mathfrak{D}_K$-arrow from $k$ to $f$ and $g_1$ is a $\mathfrak{D}_K$-arrow from $k$ to $g$  and the following diagram
	$$
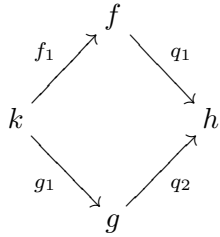
\xymatrix{&f\ar[dr]^{q_1}&\\
		k \ar[ur]^{f_1} \ar[dr]_{g_1}&&h\\
		&g\ar[ur]_{q_2}&
	}$$
	
	commutes. This means that the category $\mathfrak{D}_K$ has the  amalgamation property.
	
	In order to show that $\mathfrak{D}_K$ is directed, let $f\colon K\to X$ and 
	$g\colon K\to Y$ be a uniformly continuous mappings from $K$ to  discrete spaces.
	Consider $X\times Y$ with discrete topology and let $h=(f,g)$. Since $f,g$ are uniformly continuous and $h^{-1}(x,y)=f^{-1}(x)\cap g^{-1}(y)$ for $(x,y)\in X\times Y$ and balls in ultrametric space are disjoint or comparable the map $h$ is uniformly continuous. Obviously projections $\pi_X\colon X\times Y\to X$ and $\pi_Y\colon X\times Y\to Y$ are surjection such that $\pi_X\circ h=f$ and $\pi_Y\circ h=g$, this complete the proof that $\mathfrak{D}_K$ is directed.
\end{pf}

\begin{theorem}\label{th:10}
	 There exists a continuous \fra $\lambda$-sequence in $\mathfrak{D}_K.$
\end{theorem}
\begin{pf}
	Since $\mathcal{F}_K$ is  dominating in $\mathfrak{D}_K$ 	and $|\mathcal{F}_K|\loe \lambda$ there exists a continuous \fra $\lambda$-sequence in $\mathfrak{D}_K$ by Theorem \ref{exists}.
\end{pf}

We say that a subset $A$ of a $\lambda$-ultrametric space $(X,d)$ of weight $\kappa$  is \textit{uniformly nowhere dense} if for every $\alpha<\lambda$ there is $\beta>\alpha$ such that 
$$\{B_\beta(a) \colon B_\beta(a)\cap A=\emptyset\;\&\; a\in B_\alpha(b)\}\ne\emptyset$$ for every $b\in X.$ Note that every uniformly nowhere dense subset of the  ultrametric space $\kappa^\lambda$
is nowhere dense.   We define the categories $\mathfrak{M}$ and $\mathfrak{D}$ as follows. The objects of $\mathfrak{M}$ are $\lambda$-ultrametric space $X$ of weight not greater than $\kappa^{<\lambda}$ and $\lambda$-bounded  and arrows of $\mathfrak{M}$ are a uniformly continuous surjection. We define $\mathfrak{D}$ to be the full subcategory of $\mathfrak{M}$ whose objects are a discrete spaces of cardinality not greater than  $|\kappa^\gamma|$ for some $\gamma<\lambda.$

\begin{theorem}\label{fra-sequ-ult-la}
	Assume that  $\vec{\phi}=(\phi_\alpha\colon \alpha<\lambda)$  is a continuous \fra $\lambda$-sequence in $\mathfrak{D}_K$, where $\phi_\alpha\colon K\to U_\alpha$ for each $\alpha<\kappa.$   Then 
	\begin{enumerate}
		\item[{\rm(1)}] $\vec{u}=(U_\alpha\colon \alpha<\lambda)$ is a \fra sequence in the category $\mathfrak{D}$.
		\item[{\rm(2)}]The limit map $\phi_{\lambda}\colon K\to \lim \vec{u}$ has a left inverse, i.e. there is $r\colon \lim \vec{u}\to K$
		such that $r\circ \phi_{\lambda}=\ide_K$.
		\item[{\rm(3)}] The image $\phi_{\lambda}[K]\subs U_{\lambda}=\lim \vec{u}$  is uniformly nowhere dense.
	\end{enumerate}
\end{theorem}
\begin{proof}
	We prove property (3) only; properties (1) and (2) can be proved just like (1) and (2) of Theorem~\ref{fra-sequ}. Fix $\alpha<\lambda$. Let $Y=U_\alpha\times\kappa$ with the 
	discrete topology and define  maps $g\colon K\to Y$ and $p\colon Y\to U_\alpha$ as follows 
	$g(x)=(\phi_\alpha(x),0)$  for $x\in K$ and $p(a,\gamma)=a$ for $(a,\gamma)\in Y$. 
	Since the projection $p\colon Y\to U_\alpha$ is a morphism from $g$ to $\phi_\alpha$ and 
	$\vec{\phi}$ is a \fra sequence in $\mathfrak{D}_K$ there are $\beta\ge\alpha$ and 
	$h\colon U_\beta\to Y$ such that $p \circ h=u_\alpha^\beta$ and $h\circ \phi_\beta=g.$  
	Observe that $u_\beta^{-1}(h^{-1}(\{(a,\gamma)\}))\subs u_\alpha^{-1}(\{a\})$  
	and $u_\beta^{-1}(h^{-1}(\{(a,\gamma)\}))\cap \phi_{\lambda}[K]=\emptyset$ for all 
	$0<\gamma<\kappa$. This means that the set $\phi_{\lambda}[K]\subs U_{\lambda}=\lim \vec{u}$  is uniformly nowhere dense.
\end{proof}

By Theorem \ref{ultra}, every ultrametric space of weight $\kappa$ can be embedded into $\kappa^\kappa$. We shall prove that every $\lambda$-ultrametric space $(K,u)$ of weight $\kappa$  can be embedded into $\kappa^\lambda$ as a uniformly nowhere dense subset.

We say that $h\colon X\to Y$ is a \textit{uniform embedding } if $h$ is a uniformly continuous bijection from the ultrametric space $X$ onto the subspace  $h[X]$ of  the ultrametric space $Y$ and its inverse function $h^{-1}\colon h[X]\to X$ is uniformly continuous, where $h[X]$ inherits an ultrametric from $Y$.
\begin{corollary} A $\lambda$-ultrametric space $(K,u)$ of weight $\kappa$  can be uniformly embedded into $\kappa^\lambda$ as a uniformly nowhere dense subset.
\end{corollary}

\begin{pf}
	Let $(K,u)$ be a $\lambda$-ultrametric space of weight $\kappa.$ By Theorem \ref{th:10} there exists a $\lambda$-\fra sequence 
	$(\phi_\alpha:\alpha<\lambda)$ in $\mathfrak{D}_K,$ where $\phi_\alpha\colon K\to U_\alpha$ for each $\alpha<\lambda.$ Theorem \ref{fra-sequ-ult-la}(1) implies that $\vec{u}=(U_\alpha\colon \alpha<\lambda)$ is a \fra sequence in the category $\mathfrak{D}$.
	Since $\mathfrak{D}\subseteq \mathfrak{M}$ satisfy {\rm(L0)}--{\rm(L3)} the limit $U_{\lambda}=\lim{\vec{u}}$ is a $\mathfrak{M}$-generic object by Theorem \ref{fra-gen}. Therefore $U_{\lambda}$ is homeomorphic to $\kappa^\lambda.$ Using Theorem \ref{fra-sequ-ult-la}~(2), the map
	$\phi_{\lambda}\colon K\to\lim{\vec{u}}$ is a uniform embedding and by  Theorem \ref{fra-sequ-ult-la}(3) the image $\phi_{\lambda}[K]$ is uniformly nowhere dense  in $\kappa^\lambda.$
\end{pf}

\begin{theorem}\label{generic-ult-la}
	If $\map{\eta}{K}{\kappa^\lambda}$ is  uniformly embedded such that $\eta[K]$ is of weight $\kappa$ and 
	uniformly nowhere dense in the  $\lambda$-ultrametric space $\kappa^\lambda$, then $\map{\eta}{K}{\kappa^\lambda}$ is $\mathfrak{D}_K$-generic.
\end{theorem}

\begin{pf}
	It  suffices to show that $\eta$ satisfies condition (G2), i.e. discrete spaces $X,Y$ such that $|X|\leq|\kappa^{\delta_0}|$ and $|Y|\leq|\kappa^{\delta_1}|$ for some $\delta_0\leq\delta_1<\lambda$  and a surjection $f\colon X\to Y,$
	and a continuous map $b\colon K\to Y,$ and a continuous surjection $\map{g}{\kappa^\lambda}{X}$ such that $g\circ\eta=f\circ b$, there exists a continuous surjection $\map{h}{\kappa^\lambda}{Y}$ such that
	\[f\circ h=g \text{ and } h\circ \eta=b.\] 
	\amor
	$$\xymatrix@R=2.5pc@C=2.9pc{
		K\ar[d]_{b}\ar@{^(->}[r]^{\eta}&\kappa^\lambda\ar@{..>>}[dl]_(.49){h}\ar@{->>}[d]^{g}\\
		Y\ar@{->>}[r]_{f}&X}$$
	
	There is an increasing sequence $\{\gamma_\alpha:\alpha<\lambda\}\subs\lambda$ 
	such that $\bigcup \{\gamma_\alpha:\alpha<\lambda\}=\lambda$ and $\kappa^\lambda$ 
	is an inverse limit of $\{X_\alpha;\pi_\alpha^\beta;\alpha<\beta<\lambda\},$
	where each $X_\beta$ is a discrete space of cardinality $\kappa^{\gamma_\beta}$ 
	and each $\map{\pi_\alpha^\beta}{X_\beta}{X_\alpha}$ is a surjection such that
	$|(\pi_\alpha^\beta)^{-1}(x)|=|\kappa^{\gamma_\beta}|$ for every $x\in X_\alpha$ then a family
	$\Bee=\{(\pi_\beta)^{-1}(x)\colon x\in X_\beta,\; \beta<\lambda\}$ is a base for $\kappa^\lambda.$ 
	Since $g\circ\eta=f\circ b$ for each $y\in Y$ we have
	\[\eta[b^{-1}(y)]\subs g^{-1}(f(y))\]
	
	The family $\{g^{-1}(f(y))\colon y\in Y\}$ is a partition of $\kappa^\lambda$. Since $g$ is uniformly
	continuous there is $\alpha<\lambda$ such that $g^{-1}(x)=\bigcup \Bee^\alpha_x$ for each 
	$x\in X,$ where $\Bee^\alpha_x\subs \{(\pi_\alpha)^{-1}(a)\colon a\in X_\alpha\}.$ By a similar argument, since $b$ is uniformly
	continuous and  $\eta$ is a uniform embedding, we get $\gamma<\lambda$ such that for each $y\in Y$ with $b^{-1}(y)\neq\emptyset$ 
	there is a family of clopen set 
	$\Raa^\gamma_y\subs \{(\pi_\gamma)^{-1}(a)\colon a\in X_\gamma\}$ satisfying
	$$\bigcup\Raa^\gamma_y\cap \eta[K]=\eta[b^{-1}(y)]$$ 
	and $\bigcup\Raa^\gamma_y\cap \bigcup\Raa^\gamma_{y'}=\emptyset$ 
	for any distinct $y,y'\in Y$. If  $b^{-1}(y)=\emptyset$ then put $\Raa^\gamma_y=\emptyset$.
	Since $\eta[K]$ is uniformly nowhere dense there are $\beta>\alpha,\gamma$ 
	and    families $\Aaa^\beta_y,\Raa^\beta_y\subs \{(\pi_\beta)^{-1}(a)\colon a\in X_\beta\}$ for each $y\in Y$   such that $\gamma_\beta>\delta_1$ and 
	\begin{itemize}
		\item $\bigcup\bigcup\{\Aaa^\beta_y\cup\Raa^\beta_y\colon f(y)=x,\; y\in Y\}=g^{-1}(x),$ for all $x\in X$,
		\item $\bigcup\Aaa^\beta_y\cap\bigcup\Raa^\beta_y=\emptyset,$
		\item $\bigcup\Aaa^\beta_y\cap\eta[K]=\emptyset,$
		\item $|\Aaa^\beta_y|=|\kappa^{\gamma_\beta}|,$ and $\Raa^\beta_y=\emptyset$ for all $y\in Y$ such that $b^{-1}(y)=\emptyset$,
		\item $\bigcup\Raa^\beta_y\cap \eta[K]=\bigcup\Raa^\gamma_y\cap \eta[K],$
		\item $\bigcup(\Aaa^\beta_y\cup\Raa^\beta_y)\cap \bigcup(\Aaa^\beta_{y'}\cup\Raa^\beta_{y'})=\emptyset$ for any distinct $y,y'\in Y$.
	\end{itemize}
	
	We shall define a map $h$ on each clopen set $g^{-1}(x).$ Fix $x\in X.$ If $y\in f^{-1}(x)$ we put $h(a)=y$ 
	for all $a\in\bigcup\Raa^\beta_y\cup \bigcup \Aaa^\beta_y.$ 
	
	This finishes the proof.
\end{pf}

\begin{corollary} Every uniform homeomorphism of uniformly nowhere dense sets of weight $\kappa$ 
in  $\kappa^\lambda$ 
	can be extended to a uniform auto-homeomorphism of  $\kappa^\lambda$.
\end{corollary}

\begin{pf}
	Let $h\colon K\to L$ be a uniform homeomorphism between uniformly nowhere dense sets in  $\kappa^\lambda$. Denote by $\eta\colon K\hookrightarrow \kappa^\lambda$ 
	and $\eta_1\colon L\hookrightarrow \kappa^\lambda$ the inclusion mappings. By Theorem \ref{generic-ult-la}, both $\eta$ and $\eta_1\circ h$ are $\mathfrak{M}_K$-generic, therefore by uniqueness (\cite[Theorem 4.6]{KubFra}) there
	exists an $\mathfrak{M}_K$-isomorphism $H \colon  \kappa^\lambda\to\kappa^\lambda$ from $\eta$ to $\eta_1\circ h$. This means that there exists a uniform auto-homeomorphism of $\kappa^\lambda$ which extends the uniform homeomorphism $h$.
\end{pf}

\begin{corollary}Every uniformly nowhere dense set of weight $\kappa$ 
in  $\kappa^\lambda$ is a uniform retract of 	$\kappa^\lambda$.
\end{corollary}

\begin{pf}
	Let $\eta\colon K\hookrightarrow \kappa^\lambda$ denote the inclusion mapping of a uniformly nowhere dense set $K\subs \kappa^\lambda.$ We can represent $\eta\colon  K\hookrightarrow \kappa^\lambda$ as the limit of a sequence $\vec{u}=\{\phi_\alpha\colon \alpha<\lambda\}$ in the category $\mathfrak{D}_K$, where $\phi_\alpha\colon K\to U_\alpha.$ By Theorem \ref{generic-ult-la}, $\eta\colon  K\hookrightarrow \kappa^\lambda$ is $\mathfrak{D}_K$-generic, therefore by Theorem \ref{gen-fra}, $\vec{u}$ is a \fra sequence.  According to Theorem \ref{fra-sequ-ult-la}(2) there is a uniform	retraction $r\colon \kappa^\lambda\to K$. 
\end{pf}

\section{The category of $\kappa$-compact  ultrametric spaces}\label{sekcja:4}

Let $\kappa$ be an infinite cardinal number. A topological space $X$ is $\kappa$\textit{-compact}
if every open cover of $X$ has a subcover of size strictly less than $\kappa$. We say that $X$ is $\kappa$\textit{-compact  ultrametric} space whenever $X$ is   $\kappa$-ultrametric space and $\kappa$-compact.
Note that if  $X$ is $\kappa$-compact  ultrametric space then $|\{B_\alpha(a)\colon a\in K\}|<\kappa$ for every $\alpha<\kappa.$

 A cardinal $\kappa$ is \textit{weakly compact }if it is uncountable and has the Ramsey property $\kappa\to(\kappa)^2$, i.e., if $f\colon [\kappa]^2\to\{0,1\}$ then there are $i\in\{0,1\}$ and $A\in[\kappa]^\kappa$ such that $f\restriction A=\{i\}$ (see \cite{Jech}). 
 If $\kappa$ is not  weakly compact cardinal then $\kappa^\kappa$ is homeomorphic to $2^\kappa$ (see \cite[Corollary~3.4]{HN}). In the case that $\kappa$ is weakly compact $\kappa^\kappa$ is still not $\kappa$-compact but $2^\kappa$ becomes $\kappa$-compact.  Then $\kappa^{<\kappa}=\kappa$, see \cite{Jech}.  Recall some facts about $2^\kappa$.

\begin{theorem}[{\cite[Corollary~3.10]{Nyi}}]
	Let $\kappa$ be a weakly compact cardinal. A space is homeomorphic to the $\kappa$-ultrametric $2^\kappa$ if and only if it is $\kappa$-compact with no isolated points and admits a compatible spherically complete $\kappa$-ultrametric.
\end{theorem}

\begin{theorem}[\cite{MonScot}]\label{th:15}
	A cardinal number $\kappa$  is weakly compact if and only if the $\kappa$-ultrametric $2^\kappa$ is $\kappa$-compact.
\end{theorem}

\begin{proposition}\label{prop:16}
	Assume that $\kappa$ is a weakly compact cardinal. Then every nowhere dense subset of  the $\kappa$-ultrametric space $2^\kappa$ is a uniformly nowhere dense subset. 
\end{proposition}
\begin{proof}
	Let $u\colon 2^\kappa\times 2^\kappa\to\kappa+1$ be an ultrametric  and  $A\subs 2^\kappa$ be a nowhere dense set and   $\alpha<\kappa.$ The family  $\Raa=\{B_\alpha(a)\colon a\in 2^\kappa\}$ is a cover of $2^\kappa$. By Theorem \ref{th:15} $2^\kappa$ is $\kappa$-compact, so  $|\Raa|<\kappa.$ 
For each $V\in\Raa$ there exist $\beta_V<\kappa$ and  a set $B_V\in \{B_{\beta_V}(x)\colon x\in 2^\kappa\}$ such that 
$B_V\subset V$ and $B_V\cap A=\emptyset.$ Let $\beta=\max\{\beta_V\colon V\in\Raa\}.$ 
Since $|\Raa|<\kappa$ and $\kappa$ is regular cardinal we get $\beta<\kappa$. Therefore we get
$$\{B_\beta(a) \colon B_\beta(a)\cap A=\emptyset\; , a\in 2^\kappa,a\restriction\alpha=b\restriction\alpha\}\ne\emptyset$$ for every $b\in 2^\kappa.$
\end{proof}

\begin{proposition}\label{prop:17}
Assume that $(X,u)$	is  a $\kappa$-compact   ultrametric space and  $(Y,d)$ is a $\tau$-ultrametric space, where $\tau$ is an arbitrary ordinal. If $\kappa$ is a regular cardinal then every continuous map $f\colon X\to Y$  is uniformly continuous. 
\end{proposition}
\begin{proof}
Fix $\beta<\tau$.	Let $\Raa_1\subs\{B^u_\alpha(a)\colon a\in X,\; \alpha<\kappa\}$ be a cover of $X$ which  is a refinement of $\{f^{-1}(B^d_\beta(f(x)))\colon x\in X\}$, i.e.  for each $G\in \Raa_1$ there is $x\in X$ with $G\subs f^{-1}(B^d_\beta(f(x)))$. Since $X$ is $\kappa$-compact there is a subcover $\Raa_2\subs\Raa_1$ of size $\lambda<\kappa.$ Let $\Raa_2=\{B^u_\alpha(a)\colon a\in A, \;\alpha\in C\}$, where 
$A\subs X$ and $C\subs \kappa$ are subsets of cardinality strictly less than $\kappa$. Since $\kappa$ is a regular cardinal, the ordinal $\gamma=\sup C$ is strictly less than $\kappa$. Therefore there is a refinement $\Raa\subs \{B^u_\alpha(a)\colon a\in X,\; \alpha<\kappa\}$ of $\Raa_1$ which  is a cover of $X$, this means that  $f$ is uniformly continuous.
\end{proof}

\textbf{In the rest of this section we assume that  $\kappa$ is a weakly compact cardinal.}

 Fix a $\kappa$-compact   ultrametric space $(K,u)$ of weight not greater than
 $\kappa$, where  $u\colon K\times K\to \kappa+1.$ We define categories $\mathfrak{M}^w_K$, $\mathfrak{D}^w_K$, and a family of arrows $\mathcal{F}^w_K$ 
 as follows. 
 
 The objects of $\mathfrak{M}^w_K$ are continuous mappings $f\colon  K\to X$, where $(X,d)$ is a  $\kappa$-ultrametric space of weight 	not greater than $\kappa$. Given two $\mathfrak{M}^w_K$-objects 
 $f_0\colon  K \to X_0$, $f_1 \colon  K \to X_1$, a $\mathfrak{M}^w_K$-arrow from $f_1$ to $f_0$ is
 a continuous surjection $q \colon  X_1\twoheadrightarrow X_0.$ The composition in $\mathfrak{M}^w_K$ is the usual composition of mappings.  
 
 We define $\mathfrak{D}^w_K$ to be the full
 subcategory of $\mathfrak{M}^w_K$ whose objects are those $f \colon  K \to X$ where $X$ is a discrete spaces of cardinality	$<\kappa$.
 
 We define a family of arrows $\mathcal{F}^w_K$   similarly to Section~\ref{sekcja:3}. 
 For each ordinal $\alpha<\kappa$ and each cardinal $\xi<\kappa$  such that  $|\{B_\alpha(a)\colon a\in K\}|\leq\xi$ we define the discrete space
 $$X^\xi_\alpha = \{B_\alpha(a)\colon a\in K\} \oplus \xi$$
 and define $f^\xi_\alpha \colon K \to X_\alpha^\xi$, an object from the category $\mathfrak{D}_K$, such that $(f^\xi_\alpha)^{-1}[B_\alpha(a)] = B_\alpha(a)$ for every $a \in K$.
 Given cardinals $\xi<\delta<\kappa$, fix a map $r^\delta_\xi\colon \delta\to\xi\times\{0,1\}$ such that $|(r^\delta_\xi)^{-1}(\alpha,i)|=\delta$ for every
 $(\alpha,i)\in \xi\times\{0,1\}$.
 For each ordinal $\alpha<\kappa$ and each cardinal $\xi<\kappa$  we fix a surjection 
$p^\xi_\alpha\colon \xi\to \{B_\alpha(a)\colon a\in K\},$ where 
 $|\{B_\alpha(a)\colon a\in K\}|\leq\xi$. 
 Given ordinals $\alpha \loe \beta<\kappa$ and cardinals $\xi\loe\delta<\kappa$, we define a  surjection $q^{(\beta,\delta)}_{(\alpha,\xi)} \colon  X^\delta_{\beta}\twoheadrightarrow X^\xi_{\alpha}$ as follows:
 \[q^{(\beta,\delta)}_{(\alpha,\xi)}(x)=\begin{cases}
 	B_{\alpha}(a),& \text{if }x = B_{\beta}(a),\\
 	\pi(r^\delta_\xi(x)),& \text{if }x\in\delta\text{ and }r^\delta_\xi(x)\in\xi\times\{1\},\\
 	p^\xi_{\alpha}(\pi(r^\delta_\xi(x))),& \text{if }x\in\delta\text{ and }r^\delta_\xi(x)\in\xi\times\{0\},
 \end{cases}\]
 where $\pi$ is the canonical projection from $\xi \times \{0,1\}$ onto $\xi$. Let $$\mathcal{F}_K=\left\{q^{(\beta,\delta)}_{(\alpha,\xi)}:\alpha\loe\beta<\kappa,\;\xi\loe\delta<\kappa\right\}.$$
 
 \begin{lemma}\label{l:16}
 	The family of arrows $\mathcal{F}^w_K$ is  dominating in $\mathfrak{D}^w_K$ 
 	and $|\mathcal{F}^w_K| \loe \kappa.$
 \end{lemma}
\begin{proof}
	Using Lemma \ref{l:2} and Proposition \ref{prop:17} we obtain the hypothesis of the lemma.
\end{proof}
Using Lemma \ref{l:16} and Propositions \ref{prop:16}, \ref{prop:17}  one can  mimic the considerations of Section~\ref{sekcja:3}, to get  the following.
\begin{theorem}\label{th:17}
	Assume that  $\vec{\phi}=(\phi_\alpha\colon \alpha<\kappa)$  is a continuous \fra $\kappa$-sequence in $\mathfrak{D}^w_K$, where $\phi_\alpha\colon K\to U_\alpha$ for each $\alpha<\kappa$ and $\kappa$  is a weakly compact cardinal.  Then 
	\begin{enumerate}
		\item[{\rm(1)}] $\vec{u}=(U_\alpha\colon \alpha<\kappa)$ is a \fra sequence in the category of  discrete spaces  of cardinality	less than $\kappa$.
		\item[{\rm(2)}] Then the limit map $\phi_{\kappa}\colon K\to \lim \vec{u}$ has a left inverse, i.e., there is $r\colon \lim \vec{u}\to K$
		such that $r\circ \phi_{\kappa}=\ide_K$.
		\item[{\rm(3)}] The image $\phi_{\kappa}[K]\subs U_{\kappa}=\lim \vec{u}$  is  nowhere dense. 
		\end{enumerate}
\end{theorem}

\begin{corollary} Assume that $\kappa$ is a weakly compact cardinal. Then  a  $\kappa$-compact  ultrametric space $(K,u)$ of weight not greater than
$\kappa$ can be  embedded into the $\kappa$-ultrametric space $2^\kappa$ as a nowhere dense subset.
\end{corollary}

\begin{theorem}\label{generic-com-metr}
Assume that $\kappa$ is a weakly compact cardinal.	If $\map{\eta}{K}{2^\kappa}$ is an embedding such that $\eta[K]$ is   nowhere dense in the  $\kappa$-ultrametric space $2^\kappa$, then $\map{\eta}{K}{2^\kappa}$ is a $\mathfrak{D}^w_K$-generic.
\end{theorem}
Knaster and Reichbach \cite{KnaRei} established the following theorem:

\begin{quote}
\textit{If $P$ and $K$ are closed nowhere dense subsets of the Cantor space $2^\omega$ and $f$ is a homeomorphism between $P$ and $K,$ then there exists a homeomorphism between the Cantor space extending $f.$ }
\end{quote}

Below we have a counterpart of this theorem for the $\kappa$-ultrametric space $2^\kappa$. 

\begin{corollary}[{\cite[Lemma 3.19] {Nyi}}]  Assume that $\kappa$ is a weakly compact cardinal. Then every  homeomorphism of  nowhere dense subsets of the $\kappa$-ultrametric space $2^\kappa$ 	can be extended to an auto-homeomorphism of  $2^\kappa$.
\end{corollary}

\begin{corollary} Assume that $\kappa$ is a weakly compact cardinal. Then every  nowhere dense set in the $\kappa$-ultrametric space $2^\kappa$ is a retract.
\end{corollary}

\end{document}